\documentclass[reqno]{amsproc}

\usepackage{amsmath,amssymb,amsthm,amsfonts,latexsym}
\usepackage{graphicx}
\usepackage{color}
\usepackage{stackrel}

\usepackage{txfonts}

\theoremstyle{plain}
\newtheorem{thm}{Theorem}
\newtheorem{pro}[thm]{Proposition}
\newtheorem{lem}[thm]{Lemma}
\newtheorem{cor}[thm]{Corollary}

\theoremstyle{definition}

\newtheorem*{exaa}{{\it Example}}

\theoremstyle{remark}
\newtheorem{rem}[thm]{{\it Remark}}

\newtheorem*{prf}{{\it Proof}}

\newcounter{dc}

\makeatletter
\newcommand\footnoteref[1]{\protected@xdef\@thefnmark{\ref{#1}}\@footnotemark}
\newcommand*{\im}{{\mathfrak{Im}}}
\newcommand*{\re}{{\mathfrak{Re}}}

 \DeclareMathOperator{\okr}{{\stackrel{{\scriptscriptstyle{\mathsf{def}}}}{=}}}
  \DeclareMathOperator{\D}{d\!} 
    
 \DeclareMathOperator{\E}{e} \DeclareMathOperator{\I}{i}

\def\funkc#1#2#3#4#5{#1\colon#2\ni#3\mapsto#4\in#5}
\def\Ge{\geqslant}

\def\liczp#1{{${#1}^{\text {\rm o}}$}}

\def\pnat{^\natural}
\def\sbar#1{\,\overline{\!#1}}
\def\ulamek#1#2{\mbox{\normalfont$\frac{#1}{#2}$}}

\newcommand*\ddc{\mathcal D}

\newcommand*\hhc{\mathcal H}
\newcommand*\llc{\mathcal L}

\newcommand*\ccb{\mathbb C}
\newcommand*\nnb{\mathbb N}

\newcommand*\rrb{\mathbb R}

\begin{document}

\title[Holomorphic Hermite polynomials in two variables]
{Holomorphic Hermite polynomials in two  variables}

\author{Katarzyna G\'{o}rska}
\address{H. Niewodnicza\'{n}ski Institute of Nuclear Physics, Polish Academy of Sciences, Division of Theoretical Physics, ul. Eliasza-Radzikowskiego 152, PL 31-342 Krak\'{o}w, Poland}
\email{katarzyna.gorska@ifj.edu.pl}

\author{Andrzej Horzela}
\address{H. Niewodnicza\'{n}ski Institute of Nuclear Physics, Polish Academy of Sciences, Division of Theoretical Physics, ul. Eliasza-Radzikowskiego 152, PL 31-342 Krak\'{o}w, Poland}
\email{andrzej.horzela@ifj.edu.pl}

\author{Franciszek Hugon Szafraniec}
\address{Instytut Matematyki, Uniwersytet Jagiello\'{n}ski, ul. \L ojasiewicza 6,
30 348 Krak\'ow, Poland}
\email{umszafra@cyf-kr.edu.pl}
  \thanks{The work of third author {was} supported by the grant of NCN (National Science Center, Poland), decision No. DEC-2013/11/B/ST1/03613.
}
\begin{abstract}
Generalizations of the Hermite polynomials to many variables and/or to the complex domain have been located in mathematical and physical literature {for some decades.} Polynomials traditionally called complex Hermite ones are mostly understood as polynomials in $z$ and $\sbar z$ which in fact makes them polynomials in two real variables with complex coefficients. The present paper proposes {to investigate} for the first time \underbar{holomorphic} Hermite polynomials in two  variables. Their algebraic and analytic properties are developed here. While the algebraic properties do not differ too much for those considered so far, their analytic features are based on a kind of non-rotational orthogonality invented by van Eijndhoven and Meyers. Inspired by their invention we merely follow the idea of Bargmann's seminal paper (1961) giving explicit construction of reproducing kernel Hilbert spaces based on those polynomials. ``Homotopic" behavior of our new formation culminates in comparing it {to the {very classical} Bargmann space of two variables on one edge and the aforementioned Hermite polynomials in $z$ and $\sbar z$ on the other}. Unlike in the case of Bargmann's basis our Hermite polynomials are not product ones but factorize to it  when bonded together with the first case of limit properties leading both to the Bargmann basis and suitable form of the reproducing kernel. Also in the second limit we recover standard results obeyed by Hermite polynomials in $z$ and $\sbar z$. 
\end{abstract}
\maketitle

\subsection*{Hermite polynomials in two variables - a gentle introduction} The term ``complex Hermite polynomials" needs to be put across as it creates some disagreement among people who use it, which may impact its applicability in turn. In order to make the situation crystal clear  we intend to separate carefully a ``polynomial" as an \underbar{algebraic} object from that of being a \underbar{function}; the latter becomes important when dealing with analytic properties of polynomials like e.g. orthogonality.

The finest way of thinking of a polynomial (say in two variables as this concerns us here) is to consider it as an array of its (would-be) coefficients, which in the present situation is just a finite array of doubly indexed real or complex numbers $(a_{m,n})_{m,n=0}^{N}$ (cf. \cite[page 3 of Chapter 2]{ash} or  \cite[p 140 and 151]{fhsz:hunger}). More intuitive way of writing a polynomial we have in mind is
\begin{equation}\label{eq1}
A(X_{1},X_{2})=\sum_{m,n=0}^{N}a_{m,n}X_{1}^{m}X_{2}^{n}
\end{equation}
where $X_{1},X_{2}$ are placeholders; they are  such in the commonly used notation $\rrb[X_{1},X_{2}]$ and  $\ccb[X_{1},X_{2}]$ as well.

Consider two \underbar{polynomial} \underbar{expressions} in the sense of \eqref{eq1}
\begin{equation}\label{eq2}
H_{m, n}(X_{1}, X_{2}) \okr \sum_{k=0}^{\min\{m, n\}} \binom{m}{k} \binom{n}{k} (-1)^k k! X_{1}^{m-k} X_{2}^{n-k}.
\end{equation}
and
\begin{equation}\label{eq3}
H_{m, n}\pnat(X_{1}, X_{2}) =  \sum_{k=0}^{\min\{m, n\}} \sum_{i=0}^{m-k} \sum_{j=0}^{n-k} \frac{m! n!}{k! i! j!} \frac{\I^{m+k-i-j}}{(m-k-i)!(n-k-j)!} X_{1}^{n-k-j+i} X_{2}^{m-k-i+j}.
\end{equation}
The above  two algebraic objects are \underbar{different}: this in \eqref{eq2} is a member of $\rrb[X_{1},X_{2}]$ while that of \eqref{eq3} belongs to $\ccb[X_{1},X_{2}]$. This is why referring under our circumstances to polynomials  we would like to drop the word ``complex" whenever this is possible, in the future the word ``holomorphic" does the job instead. Needless to say that formal polynomial expressions can be used as long as \underbar{algebraic} properties (like three terms recurrence relations or even exponential generating functions) are dealt  with.

However, if one changes the point of view passing to  \underbar{polynomial} \underbar{functions} rather than dealing with formal \underbar{polynomial} \underbar{expressions}, the resulting polynomials \eqref{eq2} and \eqref{eq3} are related each to the other. This reshaping allows us in particular to handle \underbar{analytic} properties of polynomials like orthogonality and so forth. More precisely, waking up  the placeholders (``indeterminates") $X_{1},X_{2}$ in \eqref{eq2} and \eqref{eq3} and replacing them by ``variables" we come up to two polynomial functions
\begin{equation*}
H_{m,n}\colon\ccb^{2}\ni(z_{1},z_{2})\to H_{m,n}(z_{1},z_{2})\in\ccb, \quad H_{m,n}\pnat\colon\ccb^{2}\ni(u,v)\to H_{m,n}\pnat(u,v)\in\ccb,
\end{equation*}
where
\begin{equation}\label{eq4}
H_{m, n}(z_{1}, z_{2}) \okr \sum_{k=0}^{\min\{m, n\}} \binom{m}{k} \binom{n}{k} (-1)^k k! z_{1}^{m-k} z_{2}^{n-k}
\end{equation}
and
\begin{equation}\label{eq5}
H_{m, n}\pnat(u, v) =  \sum_{k=0}^{\min\{m, n\}} \sum_{i=0}^{m-k} \sum_{j=0}^{n-k} \frac{m! n!}{k! i! j!} \frac{\I^{m+k-i-j}}{(m-k-i)!(n-k-j)!} u^{n-k-j+i} v^{m-k-i+j}.
\end{equation}
They are linked by
\begin{equation*}
 H_{m,n}\pnat(u,v)=H_{m,n}(u+\I v,u-\I v),\quad H_{m,n}(z_{1},z_{2})=H_{m,n}\pnat(\ulamek{z_{1}+z_{2}}2,\ulamek{z_{1}-z_{2}}{2\I}).
\end{equation*}
If $u$ and $v$ are real then\,\footnote{\;The proof of their equivalence is shifted to Appendix, p. \pageref{A1}.} 
\begin{equation}\label{eq6}
 H_{m,n}\pnat(u,v)=H_{m,n}(u+\I v,u-\I v)=H_{m,n}(z,\sbar z)\quad\text{ with $z=u+iv$}\,;
\end{equation}
the polynomials \eqref{eq6} are pretty often referred to as Ito's polynomials after \cite{ito}, but are also appearing in the literature under the names of complex or 2D Hermite polynomials \cite{ali_{arx}, 
abg, KGorska10, dzielo, {dunkl},Fan02, fan2, fan1, ghanmi1, ghanmi2, Hu08, ismail1, ismail2,MEHIsmail2016,ismailegipt,ismailzhang, 
Lv15, AWunsche1998, AWunsche2015,xu}. \label{list-cite}
The major object of our investigation are the {Hermite} polynomials $H_{m,n}$ as defined in \eqref{eq4}, as well as consequences of the orthogonality relations  which {the \underbar{polynomial} \underbar{functions} $(z_{1},z_{2})\to H_{m,n}(z_{1},z_{2})$} satisfy with respect to some measures which are \underbar{not}  product ones. This differs our research  from that presented in the just quoted papers which, at least when orthogonally becomes concerned, sooner or later make polynomials in question functions of  real variables $u$ and $v$. Therefore inserting into a polynomial $H_{m,n}$ the variables $z$ and $\sbar z$ seems to be a rather artificial
ornamentation, or an unfortunate notational shortcut, resulting in  misleading associations\,\footnote{\;{Clarifying these accusations let us point out that a polynomial in two variables when fixing for the ``independent" variables $z$ and $\sbar z$ becomes a function of $z$ exclusively (not of $z$ and $\sbar z$). Where are the ``complex polynomials" after all? This is we abandon the term ``complex polynomials".}}. 
As long as algebraic properties\,\footnote{\;Differentiation as applied to polynomials can also be thought of as an algebraic property.}  are of our concern, in  Section \ref{s1.20.03} for instance, it does not matter whether the polynomials $H_{m,n}$ are  in $(z_{1},z_{2})$ or $(z,\sbar z)$, but it \underbar{does} when passing to their analytic properties like in the Section \ref{p1.30.11}. Because our ultimate goal is to consider and apply the Segal-Bargmann transform  focusing  exclusively on holomorphic polynomials  \eqref{eq4} forces us to abandon the other option.

A reputable recommendation for interested readers to learn more on polynomials in several variables is, besides the monograph \cite{dunkl}, the influential survey article \cite{tom}; to see how subtle algebra (and algebraic geometry) is behind polynomials of several variables take a look at \cite{3czlony}.

\section{Algebraic properties}\label{s1.20.03}
For reader's convenience we {itemize}
below properties which in this or another form can be found in the papers {quoted on p. \pageref{list-cite}. This is legitimate because only algebraic properties of the polynomials $H_{m,n}$ are concerned here.}

\subsection*{Generating function and Rodrigues formula} Each of these two usually determines orthogonal polynomials in a single variable, it also happens in two variable case, cf. \cite{ismailegipt} and references therein.
\subsubsection*{Generating function}
The Hermite polynomials $H_{m, n}$, $m,n=0,1\ldots$, defined by \eqref{eq4} may come from the generating function (cf. \cite[Th. 12]{ito} or \cite[formula (4.2)]{AWunsche2015})
\begin{equation}\label{eq7}
\sum_{m, n=0}^{\infty} \frac{s^m t^n}{m! n!} H_{m, n}(z_{1}, z_{2}) = \E^{z_{1} s + z_{2} t - st},\quad z_{1},z_{2}\in\ccb,
\end{equation}
{which} factorizes as 
\begin{equation}\label{eq8}
\exp(z_{1} s + z_{2} t - st) = \exp\left(2 \ulamek{z_{1}}{\sqrt{2}} \ulamek{s}{\sqrt{2}} - \ulamek{s^{2}}{2}\right) \exp\left(2 \ulamek{z_{2}}{\sqrt{2}} \ulamek{t}{\sqrt{2}} - \ulamek{t^{2}}{2}\right) \exp\big[\ulamek{1}{2}(s-t)^{2} \big].
\end{equation}
If $H_{m,n}$ were product polynomials, that is if $H_{m,n}(z_{1},z_{2})=H_{m}(z_{1})H_{n}(z_{2})$ with the Hermite polynomials in a single variable on the right hand side, then its generating function would compose exclusively of the product of the two first factors in \eqref{eq8}.

\subsubsection*{Rodrigues formula} The formula is as follows
\begin{equation}\label{eq9}
H_{m, n}(z_1, z_2) = (-1)^{m+n} \E^{z_1 z_2} \frac{\partial^{m+n}}{\partial z_1^n\, \partial z_2^m} \E^{-z_1 z_2}.
\end{equation}
Moreover a kind of {\em partial} generating formulae (cf. \cite{ghanmi2}) 
\begin{equation*}
\sum_{m=0}^{\infty} \frac{s^m}{m!} H_{m, n}(z_1, z_2) = (z_2 -s)^n \E^{z_1 s},\quad
\sum_{n=0}^{\infty} \frac{t^n}{n!} H_{m, n}(z_1, z_2)  = (z_1 - t)^m \E^{z_2 t} 
\end{equation*}
hold and so do {\em partial} Rodrigues formulae
\begin{equation*}
H_{m, n}(z_1, z_2) = (-1)^m \E^{z_1 z_2} \frac{\partial^m}{\partial z_2^m} (z_2^n \E^{-z_1 z_2}), \quad
H_{m, n}(z_1, z_2) = (-1)^n \E^{z_1 z_2} \frac{\partial^n}{\partial z_1^n} (z_1^m \E^{-z_1 z_2})
\end{equation*}
with a straightforward consequence, due to the Leibniz rule, 
\begin{equation*}
\frac{\partial^{n+m}}{\partial z_{1}^{n} \partial z_{2}^{m}} (-z_{1} \E^{-z_{1} z_{2}}) =  -n \frac{\partial^{n-1+m}}{\partial z_{1}^{n-1} \partial z_{2}^{m}} \E^{-z_{1} z_{2}} - z_{1} \frac{\partial^{n+m}}{\partial z_{1}^{n} \partial z_{2}^{m}} \E^{-z_{1} z_{2}}.
\end{equation*}
\subsection*{{\em Raising} and {\em lowering} operational formulae}
Such formulae read 
\begin{align}\label{eq10}
H_{m+1, n}(z_{1}, z_{2}) = (z_{1} - \partial_{z_{2}}) H_{m, n}(z_{1}, z_{2}), &\quad H_{m, n+1}(z_{1}, z_{2}) = (z_{2} -\partial_{z_{1}}) H_{m, n}(z_{1}, z_{2}),  \\
\partial_{z_{2}} H_{m, n}(z_{1}, z_{2}) = n H_{m, n-1}(z_{1}, z_{2}), &\quad \partial_{z_{1}} H_{m, n}(z_{1}, z_{2}) = m H_{m-1, n}(z_{1}, z_{2}), \label{eq11}
\end{align}
and can be found in \cite[Eq. (1.5)]{AWunsche2015} and \cite[Eqs. (14)]{KGorska10}.
They imply  known {\em recurrence relations}, see \cite[Eq. (1.6)]{AWunsche2015} or \cite[Eq. (2.1)]{MEHIsmail2016}:
\begin{align}\label{eq12}\begin{split}
H_{m+1, n}(z_{1}, z_{2}) &= z_{1} H_{m, n}(z_{1}, z_{2}) - n H_{m, n-1}(z_{1}, z_{2}), \\
H_{m, n+1}(z_{1}, z_{2}) &= z_{2} H_{m, n}(z_{1}, z_{2}) - m H_{m-1, n}(z_{1}, z_{2}). \end{split}\,
\end{align}
Formulae \eqref{eq10}, \eqref{eq11}, and \eqref{eq12} will become  important in our further consideration.

\subsection*{The Hermite polynomials in two variables  by means of classical orthogonal polynomials} 
The polynomials $H_{m,n}$ defined in \eqref{eq4} can be represented in the form 
\begin{equation}\label{eq13}
H_{m, n}(z_1, z_2) = 2^{-(m+n)} \sum_{k=0}^m\sum_{l=0}^n \binom{m}{k} \binom{n}{l} \I^{m-k} (-\I)^{n-l} H_{k+l}\big(\!\ulamek{z_1+z_2}{2}\big) H_{m+n-k-l}\big(\!\ulamek{z_1-z_2}{2\I}\big),
\end{equation}
where $H_{n}(u)$'s are the standard Hermite polynomial in a single variable \cite[(5.1)]{ali_{arx}} or \cite[(3.4)]{ismailegipt}. $H_{m, n}(z_1, z_2)$ given by \eqref{eq13} are the same as \cite[(5.1)]{ali_{arx}} for $z_{1}=z=x+\I y$ and $z_{2} = \sbar{z}$. 

{\sc Proof of \eqref{eq13}.} The substitution $s = (x+y)/\!\sqrt{2}$ and $t = (x-y)/\!\sqrt{2}$ in
\begin{equation*}
H_{m, n}(z_1, z_2) = \frac{d^m}{d s^m} \frac{d^n}{d t^n} \E^{z_1 s + z_2 t - st}\big |_{s=0,t=0}.
\end{equation*}
allows to perform the calculation 
\begin{align}\label{eq14}
\begin{split}
H_{m, n}(z_1, z_2) & = \left[2^{-\frac{n+m}{2}} \left(\frac{d}{d x} + \frac{d}{d y}\right)^m \left(\frac{d}{d x} - \frac{d}{d y}\right)^n \E^{\frac{z_1 + z_2}{\sqrt{2}} x - \frac{x^2}{2}} \E^{\frac{z_1 - z_2}{\sqrt{2}} y + \frac{y^2}{2}}\right]_{x=0\atop y=0} \\
& = 2^{-\frac{m+n}{2}} \sum_{k=0}^{m} \sum_{l=0}^{n} \binom{m}{k} \binom{n}{l} (-1)^{n-l} \left[\left(\frac{d}{dx}\right)^{k+l} \exp\left(\frac{z_1 + z_2}{\sqrt{2}} x - \frac{x^2}{2}\right)  \right]_{x=0} \\
&\times \left[\left(\frac{d}{dy}\right)^{m+n-k-l} \exp\left(\frac{z_1 - z_2}{\sqrt{2}} y - \frac{y^2}{2}\right) \right]_{y=0}.
\end{split}
\end{align}
Using \cite[formula (I.4.7), p.19]{dzielo}
\begin{equation*}
H_{n}(x, y) = \left(\frac{d^n}{dt^n} e^{x t + y t^2}\right)_{t=0}  = n! \sum_{r=0}^{\lfloor n/2\rfloor} \frac{x^{n-2r} y^{r}}{(n-2r)! r!} = (\I\!\sqrt{y})^n H_n \big(\!\ulamek{x}{2\I\!\sqrt{y}}\big),
\end{equation*}
 we can write \eqref{eq14} as \eqref{eq13}.
\qed

\begin{exaa}
For $z_{1} = z = x+\I y$ and $z_{2} = \sbar{z} = x-\I y$, $x, y\in\mathbb{R}$, \eqref{eq13} reads
\begin{equation}\label{eq15}
H_{m, n}(z, \sbar{z}) = 2^{-(m+n)} \sum_{k=0}^m\sum_{l=0}^n \binom{m}{k} \binom{n}{l} \I^{m-k} (-\I)^{n-l} H_{k+l}(x) H_{m+n-k-l}(y),
\end{equation}
see Eq. (2.2) in \cite{MEHIsmail2016}, {whereas } 
\begin{equation*}
H_{m, n}(z, \sbar{z}) = \sum_{k=0}^{\min\{m,n\}} \binom{m}{k} \binom{n}{k} (-1)^{k} k! (x + \I y)^{m-k} (x - \I y)^{n-k}.
\end{equation*}
Furthermore, 
\begin{equation}\label{eq16}
H_{m, n}(z, \sbar{z}) = \begin{cases}(-1)^{n} n! (x + \I y)^{m-n} L_{n}^{(m-n)}(x^{2}+y^{2}),\quad m\Ge n,
\\
(-1)^{m} m! (x - \I y)^{n-m} L_{m}^{(n-m)}(x^{2}+y^{2}), \quad m<n.
\end{cases}
\end{equation}
In the polar coordinates $z = r \E^{\I\theta}$ formula \eqref{eq16}  reads as
\begin{equation}\label{eq17}
H_{m, n}(z, \sbar{z}) = \sum_{k=0}^{\min\{m,n\}} \binom{m}{k} \binom{n}{k} (-1)^{k} k! r^{m+n-2k} \E^{\I\theta(m-n)},
\end{equation}
and  consequently

\begin{equation*}
H_{m, n}(z, \sbar{z}) =  \begin{cases}(-1)^{n} n! (r\E^{\I\theta})^{m-n} L_{n}^{(m-n)}(r^{2}),\quad m\Ge n,
\\
 (-1)^{m} m! (r\E^{-\I\theta})^{n-m} L_{m}^{(n-m)}(r^{2}), \quad m<n.
\end{cases}
\end{equation*}
\end{exaa}

\section{Hermite functions: analytic properties}\label{p1.30.11}
\subsection*{Orthogonality of van Eijndhoven--Meyers type}
{Now orthogonality of $H_{m,n}$ enters the scene. Therefore drawing attention to this we call from now on the polynomials $H_{m,n}$ (or rather the polynomial functions $(z_{1},z_{2})\to H_{m,n}(z_{1},z_{2})$) \em{holomorphic Hermite polynomials}.}

The formula which follows is crucial for proving orthogonality of the polynomials $H_{m,n}$
\begin{align}\label{eq18}
\begin{split}
2^{-(q+p)} \sum_{k=0}^{m} \sum_{l=0}^{n} \sum_{i=0}^{p} \sum_{j=0}^{q} \binom{m}{k} \binom{n}{l} \binom{p}{i} \binom{q}{j} \I^{m-k + q - j} (-\I)^{n-l + p - i} (k+l)! (n+m-k-l)! \\ \times \delta_{k+l, i+j} \delta_{n+m-k-l, q+p-i-j} = m! n! \delta_{m, p} \delta_{n, q}. 
\end{split}
\end{align}
The proof is technical and because of that we place it in Appendix, p. \pageref{A2}.

\begin{thm}\label{thm1}
For $z_{1}, z_{2} \in\mathbb{C}^{2}$ and $0 < \alpha < 1$ we have\,\footnote{\;In \eqref{eq19} and \eqref{eq21}, and on other occasions like this, $\D z_{1}$ and $\D z_{2}$ refer to planar Lebesgue measure.}. 
\begin{align}\label{eq19}
\begin{split}
\int_{\mathbb{C}^{2}} H_{m, n}(z_{1}, z_{2}) \overline{H_{p, q}(z_{1}, z_{2})} \exp[-\ulamek{1-\alpha}{4} |\sbar{z}_{2} + z_{1}|^{2} - \ulamek{1}{4}(\ulamek{1}{\alpha}-1)|\sbar{z}_{2}-z_{1}|^{2}] \frac{\D z_{1}\D z_{2}}{\pi^2} \\
= \frac{4\alpha}{(1-\alpha)^{2}} \left(\frac{1+\alpha}{1-\alpha}\right)^{m+n} m! n! \delta_{m, p} \delta_{n, q}. 
\end{split}
\end{align}
\end{thm}

\begin{proof}
Notice first that
from  \eqref{eq4} it follows immediately that
\begin{equation}
\label{eq20}
\overline{H_{m, n}(z_{1}, z_{2})} = H_{m, n}(\sbar{z}_{1}, \sbar{z}_{2}).
\end{equation}
With  $z_{1} = u + \I v$ and $z_{2} = u - \I v$, where $u, v\in\mathbb{C}$, put into \eqref{eq13} the LHS of \eqref{eq19} gives
\begin{align}\label{eq21}
\begin{split}
\text{LHS$_{\eqref{eq19}}$} & = 2^{-(n+m+p+q)} \sum_{k=0}^{m} \sum_{l=0}^{n} \sum_{i=0}^{p} \sum_{j=0}^{q} \binom{m}{k} \binom{n}{l} \binom{p}{i} \binom{q}{j} \I^{n-k + p-j} (-\I)^{q-i + m-l} \\
& \times 2\int_{\mathbb{C}} H_{k+l}(u) H_{i+j}(\sbar{u}) \E^{-\ulamek{(1-\alpha)}{4}(u+\sbar{u})^{2} + \ulamek{1}{4}(\ulamek{1}{\alpha} - 1)(u-\sbar{u})^{2}} \frac{\D u}{\pi} \\
& \times 2\int_{\mathbb{C}} H_{m+n-k-l}(v) H_{p + q - i-j}(\sbar{v}) \E^{-\ulamek{(1-\alpha)}{4}(v+\sbar{v})^{2} + \ulamek{1}{4}(\ulamek{1}{\alpha} - 1)(v-\sbar{v})^{2}} \frac{\D v}{\pi}\,.
\end{split}
\end{align}
Employing \cite[formula (0.5)]{SLLvanEijndhoven90} and \eqref{eq18} completes the proof.
\end{proof}

 The orthogonality relations \eqref{eq19} can be also shown using  the exponential generating function \eqref{eq7}. However such a proof, when performed like in \cite{SLLvanEijndhoven90}, would depend on the evaluation
\begin{equation}\label{eq22}
\Big\vert H_{m, n}(z_{1}, z_{2})\Big\vert \leq \sqrt{m! n!} \E^{|z_{1}| |z_{2}|},
\end{equation}
which would provide the crucial  argument  for  changing the integration and summation when passing from  \eqref{eq19} to  the exponential generating functions \eqref{eq7}.  Taking into account its own interest this way of arguing is developed in  Appendix, p. \pageref{aprilis}).
\subsection*{The basic Hilbert space of entire functions}
It is a right time to introduce our fundamental Hilbert space $\hhc^{(\alpha)}$ of entire functions in two variables. First let the {\em Hermite functions} $h^{(\alpha)}_{m,n}$, $m,n=0,1,\dots$, to be 
\begin{equation}\label{eq23}
h_{m,n}^{(\alpha)}(z_{1},z_{2})\okr\frac{1-\alpha}{2\sqrt{\alpha}} \left(\frac{1-\alpha}{1+\alpha}\right)^{\ulamek{m+n}{2}} \frac{1}{\sqrt{m! n!}}\E^{-\ulamek{z_{1}z_{2}}2}H_{m,n}(z_{1},z_{2}),\quad m,n=0,1,\dots,
\end{equation}
where $0 < \alpha < 1$ is a parameter. Using \eqref{eq13} we can represent $h_{m, n}^{(\alpha)}$ as 
\begin{equation*}
h_{m, n}^{(\alpha)}(z_{1}, z_{2}) = 2^{-\frac{m+n}{2}-1} \sum_{k=0}^{m} \sum_{l=0}^{n} \binom{m}{k} \binom{n}{l} \sqrt{\frac{\binom{m+n}{m}}{\binom{m+n}{k+l}}} \I^{m-k} (-\I)^{n-l}\, h_{k+l}^{(\alpha)}(\ulamek{z_{1}+z_{2}}{2}) h_{m+n-k-l}^{(\alpha)}(\ulamek{z_{1}-z_{2}}{2\I}),
\end{equation*}
where $h^{(\alpha)}_{m}$ are the van Eijndhoven-Meyers functions given in \cite[definition (2.3)]{SLLvanEijndhoven90} and exploited in \cite{analytic} and  \cite{Ali2014}.

Orthogonality of the Hermite polynomials $H_{m,n}$ given in \eqref{eq19} yields 
that  of the Hermite functions $h^{(\alpha)}_{m,n}$
\begin{equation}\label{eq24}
\int_{\mathbb{C}^2} h^{(\alpha)}_{m, n}(z_{1},z_{2})\, \overline{{h}^{(\alpha)}_{p, q}(z_{1},z_{2})}\E^{\ulamek{1+\alpha^{2}}{4\alpha}(z_{1}z_{2} + \sbar{z}_{1}\sbar{z}_{2}) - \ulamek{1-\alpha^{2}}{4\alpha}(z_{1}\sbar{z}_{1} + z_{2}\sbar{z}_{2})}  \frac{ \D z_{1} \D z_{2}}{\pi^2} =   \delta_{m, p} \delta_{n, q}.
\end{equation}
Define $\hhc^{(\alpha)}$ as the space of those entire functions in two variables which are in $\llc^{2}(\mu_{\alpha})$ with
\begin{equation}\label{eq25}
\mu_{\alpha}\okr g_{\alpha}\frac{\D z_{1} \D z_{2}}{\pi^2}, \quad g_{\alpha}(z_{1},z_{2})\okr \E^{\ulamek{1+\alpha^{2}}{4\alpha}(z_{1}z_{2} + \sbar{z}_{1}\sbar{z}_{2}) - \ulamek{1-\alpha^{2}}{4\alpha}(z_{1}\sbar{z}_{1} + z_{2}\sbar{z}_{2})},\; \quad z_{1},z_{2}\in\ccb.
\end{equation}
Notice that:
\begin{itemize}
\item{the exponential density $g_{\alpha}$ in \eqref{eq24} and \eqref{eq25} does not factorize into terms depending on $z_{1}$ and $z_{2}$ separately but pay attention that the second factor which it splits in is a germ of the Gausian density in two complex dimensions,}
\item{$g_{\alpha}$ is positive because so are both factors therein (the exponent is real).}
\end{itemize}
\begin{pro}\label{t1.11.05}
$\hhc^{(\alpha)}$ is a closed subspace of $\llc^{2}(\mu_{\alpha})$, hence it is  an autonomous Hilbert space with the norm inherited from  that of $\llc^{2}(\mu_{\alpha})$. $\hhc^{(\alpha)}$ has the reproducing kernel.
\end{pro}
{\sc Proof.}
Apply \cite[Corrolary 1, Proposition 3, p. 33]{wuj2} (detailed proof based on the Cauchy integral formula on polydiscs can be found in \cite[Corollary 1.4, p. 57]{wuj}).
What one gets from that can be  itemised as 
\begin{enumerate}
\item[\liczp 1] an $\llc^{2}(\mu_{\alpha})$ convergent sequence is uniformly convergent on any compact subset of $\ccb^{2}$; this implies $\hhc^{(\alpha)}$ to be closed in $\llc^{2}(\mu_{\alpha})$,
\item[\liczp 2] the space $\hhc^{(\alpha)}$ has the bounded point evaluation property, hence it has a reproducing kernel (to be determined later).\qed
 \end{enumerate}

\begin{rem}It is worthy to notice  that\,\footnote{\;Whenever a double infinite summation appears we understand it in the sense \cite[par. 4.15, p. 83]{rudin}. This is equivalent to summability with respect to the counting measure on the index set $\nnb^{2}$ in this case, or to unconditional convergence of the series in question. As a consequence we can use Fubini's Theorem rules for interchanging summation, cf. \cite[Theorem 8.8, p. 164]{rudin}. }\label{t1.12.05}
\begin{equation}\label{eq26}
 \sum_{m,n=0}^{\infty}  |{h}^{(\alpha)}_{m, n}(z_{1}, z_{2})|^{2} = \frac{(1-\alpha^{2})^{2}}{16\,\alpha^{2}} \E^{-\ulamek{1+\alpha^2}{4\alpha}(z_{1}z_{2} + \sbar{z}_{1}\sbar{z}_{2}) + \ulamek{1-\alpha^2}{4\alpha}(z_{1}\sbar{z}_{1} + z_{2}\sbar{z}_{2})} < +\infty,
\end{equation}
which proof is shifted to Appendix, p. \pageref{rem6}. The formula \eqref{eq26} allows us to think of the reproducing kernel which we are going to introduce soon without any hesitation. 
\end{rem}

\begin{rem}\label{t1.15.08}
Though one may try to determine the reproducing kernel for $\hhc^{(\alpha)}$ from \liczp 2 via the Riesz representation theorem we prefer to use the formula of Zaremba \cite[formulae (1.8) and (1.11)]{wuj1}, also \cite[p. 9, formula (3.1)]{wuj}. This preference is emboldened by having at hand  all the ingredients already done. Though this looks like a kind of detour it makes us sure no logical argument is missing. This way we avoid also any demand of proving completeness of $(h^{(\alpha)}_{m,n})_{m,n=0}^{\infty}$, it is given as a byproduct.
\end{rem}

\noindent Going on pick up the kernel
\begin{equation}\label{eq27}
K^{(\alpha)}(z_1, z_2, w_1, w_2) \okr \sum_{m, n=0}^{\infty} {h}^{(\alpha)}_{m, n}(z_1, z_2) \overline{{h}^{(\alpha)}_{m, n}(w_1, w_2)},\quad z_{1},z_{2},w_{1},w_{2} \in \ccb,
\end{equation}
which, due to \eqref{eq26}, is well defined.

\begin{lem}\label{thm3}
For the kernel $K^{(\alpha)}$ defined by \eqref{eq27} we have
\begin{equation}\label{eq28}
K^{(\alpha)}(z_1, z_2, w_1, w_2) = \frac{(1-\alpha^2)^2}{16\,\alpha^2} \exp\left[- \frac{1+\alpha^2}{4\alpha} (z_{1}z_{2} + 
\sbar{w}_{1}\sbar{w}_{2}) + \frac{1-\alpha^2}{4\alpha}(z_1 \sbar{w}_1 + z_2 \sbar{w}_2)\right].
\end{equation}
\end{lem}
\begin{proof}
Using \eqref{eq23} we get 

\begin{equation*}
K^{(\alpha)}(z_1, z_2, w_1, w_2) = \frac{(1-\alpha)^{2}}{4 \alpha} \E^{-\ulamek{1}{2}(z_{1}z_{2} + \sbar{w}_{1}\sbar{w}_{2})} \sum_{m, n=0}^{\infty} \frac{1}{m! n!} \left(\frac{1-\alpha}{1+\alpha}\right)^{m+n}\! H_{m, n}(z_{1}, z_{2}) H_{m, n}(\sbar{w}_{1}, \sbar{w}_{2}).
\end{equation*}
Formula \eqref{eq28} can be derived using  \cite[formula (5.2), p. 2150]{AWunsche2015} for $s=t=(1-\alpha)/(1+\alpha)$, $z = z_{1}$, $z' = z_{2}$, $w = \sbar{w}_{2}$, and $w' = \sbar{w}_{1}$.
\end{proof}
	\begin{thm}\label{t2.12.05}
	$\hhc^{(\alpha)}$ is a reproducing kernel Hilbert space with the kernel given by \eqref{eq28}.
		\end{thm}
		
		\begin{proof}
		From \eqref{eq27} we can infer that the function
		\begin{equation*}
		\funkc \varPhi{\ccb^{2}}{(z_{1},z_{2})}{K^{(\alpha)}(z_{1},z_{2},z_{1},z_{2})}\ccb
		\end{equation*}
is in $\llc^{1}(\mu_{\alpha})$.	Let $\widehat \hhc_{(\alpha)}$	be the reproducing kernel Hilbert space corresponding to the kernel $K^{(\alpha)}$ given by \eqref{eq27}. It is obvious that $\widehat \hhc^{(\alpha)}\subset \hhc^{(\alpha)}$. To show that these two spaces coincide take $\xi\okr(\xi_{m,n})_{m,n=0}^{\infty}\in\ell^{2}(\nnb^{2})$ such that $\sum_{m,n=0}^{\infty}
\xi_{m,n}h^{(\alpha)}_{m.n}(z_{1},z_{2})=0,$ for all $z_{1},z_{2}\in\ccb$. This gives us
\begin{equation}\label{eq29}
|\sum_{m,n=0}^{\infty}
\xi_{m,n}h^{(\alpha)}_{m.n}(z_{1},z_{2})|^{2}=0.
\end{equation}
Integrating \eqref{eq29} with respect to $\mu_{\alpha}$ and interchanging summation with integration (Lebesgue's dominated convergence theorem with the domination $\varPhi$ makes it possible) gives, after applying orthonormality \eqref{eq24}, that the  right hand side of \eqref{eq29} reduces to
$$
\sum_{m,n=0}^{\infty}
|\xi_{m,n}|^{2}=0,
$$
which means $\xi=0$. Now \cite[Proposition 10, p. 7]{wuj1} (for details see \cite[Proposition 5.2, p. 14]{wuj}) yields the Hermite functions $h^{(\alpha)}_{m,n}$, $m,n=0,1,\ldots,$ to constitute an orthonormal family in $\widehat \hhc^{(\alpha)}$ as well; they always form {\em par force} a complete family, see the same references.

The bisequence $(h^{(\alpha)}_{m,n})_{m,n=0}^{\infty}$ is
\begin{itemize}
\item  \underbar{orthogonal} in $\hhc^{(\alpha)}$,
\item \underbar{orthonormal} in $\widehat\hhc^{(\alpha)}$, hence a Hilbert space basis of $\widehat\hhc^{(\alpha)}$.
\end{itemize}
This means that 
\begin{itemize}
\item  $\widehat\hhc^{(\alpha)}=\hhc^{(\alpha)}$;
\item the norm of $\widehat\hhc^{(\alpha)}$ coincides with that of $\hhc^{(\alpha)}$ (for functions in the previous space, of course).
\end{itemize}
As the further consequence the spaces $\widehat\hhc^{(\alpha)}$ and $\hhc^{(\alpha)}$ have the same kernel.
	\end{proof}
\begin{rem}
We  believe the way which has tackled the question is instructive; in many sources concerning reproducing kernel Hilbert spaces of holomorphic functions arguments provided for this are, so to speak, too fast.
\end{rem}

\subsection*{Relating $\hhc^{(\alpha)}$ to the Bargmann space}

In   the Bargmann space \cite{Bargmann} of two variables\,\footnote{\label{f1.19.09}\;The second entry $2$ in the subscript of $\hhc_{{\rm Barg},2}$ indicates the number of variables.} 
$$\mathcal{H}_{{\rm Barg},2} \okr {\rm Hol}(\ccb^{2})\cap{\llc^{2}}(\ccb^2, \pi^{-2} \E^{-|z_1|^2 -|z_2|^2} \D z_1\D z_2)$$ 
the sequence
\begin{equation*}
\varPhi_{m, n}(z_{1},z_{2}) \okr \frac{z_1^m}{\sqrt{m!}} \frac{z_2^n}{\sqrt{n!}}, \quad z_1, z_2 \in \mathbb{C}, \quad m, n = 0, 1, 2, \ldots,
\end{equation*}
which is an orthonormal basis, gives a rise of the Bargmann-like reproducing kernel 
\begin{equation*}
K_{{\rm Barg},2}(z_1, z_2, {w}_1, {w}_2) \okr \sum_{m=0}^{\infty}\sum_{n=0}^{\infty} \varPhi_{m, n}(z_1, z_2) \overline{\varPhi_{m, n}(w_1, w_2)} = \E^{z_1 \sbar{w}_1 + z_2 \sbar{w}_2}, \quad z_{1}, z_{2}, w_1, w_2 \in \mathbb{C}. 
\end{equation*} 
The interrelation between $\hhc^{(\alpha)}$ and $\hhc_{{\rm Barg},2}$ is going to be implemented, like in the original Bargmann approach \cite{Bargmann}, by an integral transformation with explicitly calculated  kernel.

{\sc From $\hhc^{(\alpha)}$ to $\mathcal{H}_{{\rm Barg},2}$.}
Define the mapping $U\colon {\hhc^{(\alpha)}}\to\mathcal{H}_{{\rm Barg},2}$ as
\begin{align}\label{eq37}
\begin{split}
(U h^{(\alpha)}_{m, n})(z_1, z_2) 
\okr \int_{\mathbb{C}^2} A(z_1, z_2, \sbar{w}_{1}, \sbar{w}_{2}) h^{(\alpha)}_{m, n}(w_{1}, w_{2}) \mu_{\alpha}(\D w_{1} \D w_{2})
\end{split}
\end{align}
with
\begin{equation}\label{eq38}
A(z_1, z_2, \sbar{w}_{1}, \sbar{w}_{2}) \okr \frac{1-\alpha}{2\sqrt{\alpha}} \exp\left\{-\ulamek{1}{2}\sbar{w}_{1}\sbar{w}_{2} + \lambda(z_{1}\sbar{w}_{1} + z_{2}\sbar{w}_{2}) - \lambda^2 z_1 z_2\right\}
\end{equation}
and $\lambda \okr\, [(1-\alpha)/(1+\alpha)]^{1/2}$.  Show first that
\begin{equation*}
\varPhi_{m,n}=Ah_{m,n}^{(\alpha)},\quad m,n=0,1,\ldots.
\end{equation*}
Using  the generating formula \eqref{eq7} for the Hermite polynomials $H_{m,n}$ determines that for  the Hermite functions $h_{m,n}^{(\alpha)}$
\begin{align}\label{eq39}
\begin{split}
\sum_{m, n=0}^{\infty} \varPhi_{m, n}(z_1, z_2)& \overline{h^{(\alpha)}_{m, n}(w_{1}, w_{2})} = 
\frac{1-\alpha}{2\sqrt{\alpha}} \exp\left\{-\ulamek{1}{2}\sbar{w}_{1}\sbar{w}_{2} + \lambda(z_{1}\sbar{w}_{1} + z_{2}\sbar{w}_{2}) - \lambda^2 z_1 z_2\right\}.
\end{split}
\end{align}
Multiplying \eqref{eq39} by $h_{k,l}^{(\alpha)}(w_{1}, w_{2}) \E^{\ulamek{1+\alpha^{2}}{4\alpha}(w_{1}w_{2} + \sbar{w}_{1}\sbar{w}_{2}) - \ulamek{1-\alpha^{2}}{4\alpha}(w_{1}\sbar{w}_{1} + w_{2}\sbar{w}_{2})}$ and integrating it over $\ccb^{2}$ we arrive at\,\footnote{\;Changing integration with summation on the left hand side is made possible due to the Lebesgue dominated convergence theorem arranged in a suitable way.}
\begin{equation}\label{eq40}
\varPhi_{m,n}(z_{1},z_{2})=\int_{\mathbb{C}^2} A(z_1, z_2, \sbar{w}_{1}, \sbar{w}_{2}) h_{k,l}^{(\alpha)}(w_{1}, w_{2}) \mu_{\alpha}(\D w_{1}, \D w_{2}),
\end{equation}
which means that $U$ maps orthonormal basis onto orthonormal basis. 

\begin{cor}\label{t1.15.05}
The operator $U\colon \hhc^{(\alpha)}\to\hhc_{{\rm Barg},2}$ defined by \eqref{eq37} is unitary.
\end{cor}

The function $A$ defined by formula \eqref{eq38}, which is the kernel of the integral transformation  \eqref{eq40}, satisfies the relation
\begin{equation}\label{eq41}
\int_{\mathbb{C}^2} A(w_1, w_2, \sbar{u}, \sbar{v}) \overline{A(z_1, z_2, \sbar{u}, \sbar{v})} \mu_{\alpha}(\D u, \D v) 
= \E^{w_1 \sbar{z}_1 + w_2 \sbar{z}_2}.
\end{equation}
This can be proved by inserting \eqref{eq38} into the LHS of \eqref{eq41}.

{\sc From $\mathcal{H}_{{\rm Barg},2}$ to $\hhc^{(\alpha)}$.}
Unitarity of the operator $U$ from Corollary \ref{t1.15.05} implies the existence of the  mapping $U^{-1}$: $\hhc_{{\rm Barg},2}\to \hhc^{(\alpha)}$,  which is unitary as well. Guided by the formulae \eqref{eq39}, \eqref{eq40} and \eqref{eq38} let us define
\begin{equation*}
(W f)(u, v) \okr \int_{\ccb^{2}} \overline{A(z_{1}, z_{2}, \sbar{u}, \sbar{v})} f(z_{1}, z_{2}) \E^{-\ulamek{1}{2}(|z_{1}|^{2} + |z_{2}|^{2})} \frac{\D z_{1} \D z_{2}}{\pi^{2}}.
\end{equation*}
We show that $U^{-1} = W $. Indeed, 
\begin{align*}
\begin{split}
g(w_{1}, w_{2}) & = \int_{\ccb^{2}} A(w_{1}, w_{2}, \sbar{u}, \sbar{v}) \left[\int_{\ccb^{2}} \overline{A(z_{1}, z_{2}, \sbar{u}, \sbar{v})} f(z_{1}, z_{2}) \E^{-\ulamek{1}{2}(|z_{1}|^{2} + |z_{2}|^{2})} \frac{\D z_{1} \D z_{2}}{\pi^{2}}\right] \mu_{\alpha}(\D u, \D v)\\
& = \int_{\ccb^{2}} f(z_{1}, z_{2}) \left[\int_{\ccb^{2}} A(w_{1}, w_{2}, \sbar{u}, \sbar{v}) \overline{A(z_{1}, z_{2}, \sbar{u}, \sbar{v})} \mu_{\alpha}(\D u, \D v)\right]
 \E^{-\ulamek{1}{2}(|z_{1}|^{2} + |z_{2}|^{2})} \frac{\D z_{1} \D z_{2}}{\pi^{2}} \\
& \stackrel{\eqref{eq41}}= \int_{\ccb^{2}} K_{{\rm Barg},2}(w_{1}, w_{2}, \sbar{u}, \sbar{v}) f(z_{1}, z_{2}) \frac{\D z_{1} \D z_{2}}{\pi^{2}} =  f(w_{1}, w_{2}), \quad w_{1}, w_{2} \in\ccb. 
\end{split}
\end{align*}

\section{A supplement: two limits}\label{s1.17.09}	
Results of the Section \ref{p1.30.11} 
are patterned on \cite{SLLvanEijndhoven90} and hold for $0<\alpha<1$. However, the orthogonality relations  \eqref{eq19}, as well as the expression for the reproducing kernel \eqref{eq28}, become meaningless for $\alpha=0$ and $\alpha=1$ being either identically vanishing  or  {singular}. To overcome this pathology we shall investigate what we will get if perform,  with appropriate care\,\footnote{\;This is inevitable if one tries to analyse mutual interrelations between entangled and coherent states, cf. \cite{springer}.}, the limit procedures  $\alpha\to 1-$ and $\alpha\to 0+$.

As the first step we  redesign the orthogonality relations \eqref{eq19}
\begin{align*}
\begin{split}
\int_{\mathbb{C}^{2}} H_{m, n}(z_{1}, z_{2}) \overline{H_{p, q}(z_{1}, z_{2})} {\exp[-\ulamek{1-\alpha}{4} |\sbar{z}_{2} + z_{1}|^{2} - \ulamek{1}{4}(\ulamek{1}{\alpha}-1)|\sbar{z}_{2}-z_{1}|^{2}]}\frac{\D z_{1}\D z_{2}}{\pi^2} \\
= \frac{4 \alpha}{(1-\alpha)^{2}} \left(\frac{1+\alpha}{1-\alpha}\right)^{m+n} m! n! \delta_{m, p} \delta_{n, q}. 
\end{split}
\end{align*}
Changing the variables
\begin{equation}\label{280218_1} 
z_{1}=\ulamek{1}{\sqrt{1-\alpha}}(u_1-\sqrt{\alpha}u_2), \quad\quad\quad z_{2}=\ulamek{1}{\sqrt{1-\alpha}}(\bar{u}_1+\sqrt{\alpha}\bar{u}_2)
\end{equation}
makes the exponential measure in the integrand  of the orthogonality relations independent of $\alpha$ and removes the factor $4\alpha/(1-\alpha)^2$ in its right-hand side. Defining new, now $\alpha$-dependent, polynomials
\begin{equation}\label{eq1212} 
\widetilde{H}^{(\alpha)}_{m,n}(z_1,z_2)\okr\frac{1}{\sqrt{m!n!}}\left(\frac{1-\alpha}{1+\alpha}\right)^\frac{m+n}{2}{H}_{m,n}(z_1,z_2),
\end{equation} 
we arrive at the orthogonality relations for $\widetilde{H}^{(\alpha)}_{m, n}\left(\frac{u_1-\sqrt{\alpha}u_2}{\sqrt{1-\alpha}}, \frac{\bar{u}_1+\sqrt{\alpha}\bar{u}_2}{\sqrt{1-\alpha}}\right)$ 
\begin{equation}\label{eq19bisbis}
\begin{split}
\int_{\mathbb{C}^{2}} \widetilde{H}^{(\alpha)}_{m, n}\left(\frac{u_1-\sqrt{\alpha}u_2}{\sqrt{1-\alpha}}, \frac{\bar{u}_1+\sqrt{\alpha}\bar{u}_2}{\sqrt{1-\alpha}}\right) \overline{\widetilde{H}^{(\alpha)}_{p, q}}\left(\frac{u_1-\sqrt{\alpha}u_2}{\sqrt{1-\alpha}}, \frac{\bar{u}_1+\sqrt{\alpha}\bar{u}_2}{\sqrt{1-\alpha}}\right) \\
\times{\exp[-|{u}_{1}|^{2} - |u_{2}|^{2}] }\frac{\D u_{1}\D u_{2}}{\pi^2} =\delta_{m, p} \delta_{n, q}. 
\end{split}
\end{equation}
Thus the polynomials $\widetilde{H}^{(\alpha)}_{m, n}\left(\frac{u_1-\sqrt{\alpha}u_2}{\sqrt{1-\alpha}}, \frac{\bar{u}_1+\sqrt{\alpha}\bar{u}_2}{\sqrt{1-\alpha}}\right)$ are orthonormal with respect to the standard Gaussian measure $\pi^{-2}e^{-(|{u}_{1}|^{2} + |u_{2}|^{2})}\D u_{1}\D u_{2}$ in two complex dimensions. 
\subsection*{The limits: orthogonality}
\subsubsection*{Case $\alpha \to 1-$} As we will see this limit brings us to the Bargmann space  $\hhc_{{\rm Barg},\,2}$.

\noindent The known equality (cf. \cite[p. 97]{SLLvanEijndhoven90})
\begin{equation}\label{eq51}
\lim_{t\to 0} \left(\frac{t}{2}\right)^{n} H_{n}\big(\ulamek{z}{t}\big) = z^{n}, \;\text{ uniformly on compact subsets of $\ccb$},
\end{equation}
yields
\begin{equation}\label{eq52}
\lim_{t\to 0} t^{m+n} H_{m, n}(\ulamek{z_{1}}{t}, \ulamek{z_{2}}{t}) = z_{1}^{m} z_{2}^{n}.
\end{equation}
Indeed, using the formula  \eqref{eq13} we have
\begin{equation*}
 t^{m+n} H_{m, n}(\ulamek{z_{1}}{t}, \ulamek{z_{2}}{t})= \sum_{k=0}^{m} \sum_{l=0}^{n} \binom{m}{k} \binom{n}{l} \I^{m-k} (-\I)^{n-l}  (\ulamek{t}{2})^{k+l} H_{k+l}(\ulamek{z_{1}+z_{2}}{2t})   (\ulamek{t}{2})^{m+n-k-l} H_{m+n-k-l}(\ulamek{z_{1}-z_{2}}{2\I t})
\end{equation*}
from which we get \eqref{eq52} after taking the limit \eqref{eq51} and calculating two binomial sums.

\begin{pro}\label{p-24/05/2016-2} 
The polynomials $\widetilde{H}^{(\alpha)}_{m, n}(z_{1}, z_{2})$ defined by \eqref{eq1212}  satisfy
\begin{equation}\label{eq53}
\lim_{\alpha\to 1-} \widetilde{H}^{(\alpha)}_{m, n}\left(\frac{u_1-\sqrt{\alpha}u_2}{\sqrt{1-\alpha}}, \frac{\bar{u}_1+\sqrt{\alpha}\bar{u}_2}{\sqrt{1-\alpha}}\right) = \frac{1}{\pi}\frac{\left(\ulamek{1}{\sqrt{2}}(u_{1}-u_{2})\right)^{m}}{\sqrt{m!}} \frac{\left(\ulamek{1}{\sqrt{2}}(\bar{u}_{1}+\bar{u}_{2})\right)^{n}}{\sqrt{n!}}.
\end{equation}
\end{pro}
{\sc Proof.} The above is easily seen from the definition \eqref{eq1212} and performing the limit $\alpha\to 1-$ for \eqref{eq52} with a substitution $t=\sqrt{1-\alpha}$. \qed

\begin{rem}
Taking the limit $\alpha\to 1-$ in the orthogonality relations \eqref{eq19bisbis}\;\footnote{Changing the order of taking the limit and integration is allowed because of the Lebesque dominated convergence theorem guaranteed by \eqref{eq22}.} we obtain
\begin{equation}\label{eq19bisbisbis}
\begin{split}
\int_{\mathbb{C}^{2}} 
\frac{\left(\ulamek{1}{\sqrt{2}}(u_{1}-u_{2})\right)^{m}}{\sqrt{m!}} \frac{\left(\ulamek{1}{\sqrt{2}}(\bar{u}_{1}+\bar{u}_{2})\right)^{n}}{\sqrt{n!}}\frac{\left(\ulamek{1}{\sqrt{2}}(\bar{u}_{1}-\bar{u}_{2})\right)^{p}}{\sqrt{p!}} \frac{\left(\ulamek{1}{\sqrt{2}}({u}_{1}+{u}_{2})\right)^{q}}{\sqrt{q!}}\\
\times{\exp[-|{u}_{1}|^{2} - |u_{2}|^{2}] }\frac{\D u_{1}\D u_{2}}{\pi^2} =\delta_{m, p} \delta_{n, q}. 
\end{split}
\end{equation}
which after the change of variables 
\begin{equation}\label{280218_2}
\xi_{1} = \ulamek{1}{\sqrt{2}}(u_{1}-u_{2}),\quad\quad\quad \xi_{2} = \ulamek{1}{\sqrt{2}}(\bar{u}_{1}+\bar{u}_{2})\in\ccb
\end{equation}
become
\begin{equation}\label{eq19bisbisbisbis}
\int_{\mathbb{C}^{2}} 
\frac{\xi_1^{m}}{\sqrt{m!}} \frac{{\bar \xi}_1^{p}}{\sqrt{p!}} \frac{\xi_{2}^{n}}{\sqrt{n!}}\frac{{\bar \xi}_2^{q}}{\sqrt{q!}}{\exp[-|{\xi}_{1}|^{2} - |\xi_{2}|^{2}]} \frac{\D \xi_{1}\D \xi_{2}}{\pi^2} =\delta_{m, p} \delta_{n, q}, 
\end{equation}

\noindent i.e., the limit $\alpha\to 1-$ of the orthogonality relations for the Hermite  polynomials $H_{m, n}^{(\alpha)}$ coincides with this appearing in the standard two dimensional Bargmann space $\hhc_{{\rm Barg},\,2}$.
\end{rem}
\subsubsection*{Case $\alpha \to 0+$.}  This limit brings us to the celebrated case of the  2D Hermite polynomials ${H}_{m,n}(u,{\bar u})$. Calculating the limit  of  \eqref{eq19bisbis} for $\alpha \to 0+$  gives
\begin{equation}\label{eq19bisnew}
\int_{\mathbb{C}^{2}} \widetilde{H}^{(0)}_{m, n}\left({u_1}, {\bar{u}_1}\right) \overline{\widetilde{H}^{(0)}_{p, q}}\left(u_1, \bar{u}_1\right)
{\exp[-|{u}_{1}|^{2} - |u_{2}|^{2}] }\frac{\D u_{1}\D u_{2}}{\pi^2} =\delta_{m, p} \delta_{n, q}, 
\end{equation}
where $\widetilde{H}^{(0)}_{m, n}\left({u_1}, {\bar{u}_1}\right)=\frac{1}{\sqrt{m!n!}}{H}_{m,n}(u,{\bar u})$. Factorizing the integral in \eqref{eq19bisnew} and performing the integration over $u_2$ we get the orthogonality relations for ${H}_{m,n}(u,{\bar u})$ derived in \cite{FanKlauder}, for example.

\subsection*{The limits: reproducing kernels}
Following the Zaremba approach and constructing the reproducing kernel with the use of the polynomials $\widetilde{H}^{(\alpha)}_{m,n}(z_1,z_2)$ we calculate
\begin{multline} \label{rep1}
\widetilde{K}^{(\alpha)}(z_1,z_2,w_1,w_2)=\sum\limits_{m,n=0}^{\infty}\widetilde{H}^{(\alpha)}_{m,n}(z_1,z_2)\overline{\widetilde{H}^{(\alpha)}_{m,n}}(w_1,w_2)=\\
=\sum\limits_{m,n=0}^{\infty}\frac{1}{m!n!}\left(\frac{1-\alpha}{1+\alpha}\right)^{m+n}{H}_{m,n}(z_1,z_2)\overline{{H}_{m,n}}(w_1,w_2)=
\\
=\frac{(1+\alpha)^2}{4\alpha}\exp{\left[\frac{(1-\alpha)^2}{4\alpha}(z_1z_2+{\bar w}_1{\bar w}_2)+\frac{(1+\alpha)(1-\alpha)}{4\alpha}(z_1{\bar w}_1+z_2{\bar w}_2)\right]}
\end{multline}
obtained using the Lemma 5. Introducing new variables $u_1, u_2, s_1, s_2$ and making a change 
\begin{align}\label{280218_4}
\begin{split}
&z_{1}=\ulamek{1}{\sqrt{1-\alpha}}(u_1-\sqrt{\alpha}u_2), \quad z_{2}=\ulamek{1}{\sqrt{1-\alpha}}(\bar{u}_1+\sqrt{\alpha}\bar{u}_2),
\\
&w_{1}=\ulamek{1}{\sqrt{1-\alpha}}(s_1-\sqrt{\alpha}s_2),\quad w_{2}=\ulamek{1}{\sqrt{1-\alpha}}(\bar{s}_1+\sqrt{\alpha}\bar{s}_2)
\end{split}
\end{align}
we rewrite the last line of \eqref{rep1} as
\begin{multline}\label{rep2}
\frac{(1+\alpha)^2}{4\alpha}
\exp{\left[-\frac{1-\alpha}{4\alpha}\left[\left(u_1-\sqrt{\alpha}{u}_2\right)\left({\bar u}_1+\sqrt{\alpha}{\bar u}_2\right) +\left({\bar s}_1-\sqrt{\alpha}{\bar s}_2\right)\left({s}_1+\sqrt{\alpha}{s}_2\right)\right]\right.}
\\
{\left.+\frac{1+\alpha}{4\alpha}\left[\left(u_1-\sqrt{\alpha}{u}_2\right)\left({\bar s}_1-\sqrt{\alpha}{\bar s}_2\right) +\left({\bar u}_1+\sqrt{\alpha}{\bar u}_2\right)\left({s}_1+\sqrt{\alpha}{s}_2\right)\right]\right]}.
\end{multline}
It is regular at $\alpha \to 1-$ though singular at $\alpha \to 0+$. 

\subsubsection*{Case $\alpha \to 1-$.} In this limit the expression  \eqref{rep2} becomes
\begin{multline}\label{rep3}
\exp{\left[\frac{1}{2}\left[\left(u_1-{u}_2\right)\left(\bar{s}_1-\bar{s}_2\right)+
\left({\bar u}_1+\bar{u}_2\right)\left({s}_1+{s}_2\right)\right]\right]}=
\\
=\exp{\left[\left(\frac{u_1-{u}_2}{\sqrt{2}}\right)\left(\frac{\bar{s}_1-\bar{s}_2}{\sqrt{2}}\right)+
\left(\frac{{\bar u}_1+\bar{u}_2}{\sqrt{2}}\right)\left(\frac{{s}_1+{s}_2}{\sqrt{2}}\right)\right]}
\end{multline}
which simplifies to the two dimensional Bargmann's kernel $\exp{\left(\xi_1\bar{\zeta}_1 + \xi_2\bar{\zeta}_2\right)}$ under the change of variables
\begin{equation}\label{20318_1} 
\xi_1=\ulamek{1}{\sqrt{2}}(u_{1}-u_{2}),\quad\quad\xi_2=\ulamek{1}{\sqrt{2}}(\bar{u}_{1}+\bar{u}_{2}), 
\quad\quad\zeta_1=\ulamek{1}{\sqrt{2}}(s_{1}-s_{2}),\quad\quad \zeta_2=\ulamek{1}{\sqrt{2}}(\bar{s}_{1}+\bar{s}_{2})
\end{equation} 
analogous to \eqref{280218_2} which we have used to  transform \eqref{eq19bisbisbis} into the Bargmann orthogonality relations  \eqref{eq19bisbisbisbis}. So, as just shown, performing the limit $\alpha \to 1-$ transforms the space $\mathcal{H}^{(\alpha)}$ into the product of two standard Bargmann's spaces. 

\subsubsection*{Case $\alpha \to 0+$.}  This case will be treated in a way as physicists are used to do \cite{AWunsche1998,AWunsche2015} without pretending to the details of its mathematical correctness.  As just shown the limit $\alpha \to 0+$ of \eqref{eq19bisnew} reads 
\begin{equation}\label{30318_1}
\int_{\mathbb{C}^{2}} \widetilde{H}^{(0)}_{m, n}\left({u}, {\bar{u}}\right) \overline{\widetilde{H}^{(0)}_{p, q}}\left(u, \bar{u}\right)
{\exp[-|{u}|^{2}] }\frac{\D u}{\pi} =\delta_{m, p} \delta_{n, q} 
\end{equation}
and in what follows we will investigate a possibility of constructing the reproducing kernel formed using $\widetilde{H}^{(0)}_{m, n}\left({u}, {\bar{u}}\right)$.

Using \eqref{rep1} and \eqref{280218_4} we get 
\begin{multline} \label{rep5}
\sum\limits_{m,n=0}^{\infty}\widetilde{H}^{(0)}_{m, n}\left({u}, {\bar{u}}\right)\overline{\widetilde{H}^{(0)}_{m, n}\left({s}, {\bar{s}}\right)}=\lim\limits_{\alpha\to 0+}\widetilde{K}^{(\alpha)}(u_1,u_2,s_1,s_2)=
\\
=\lim\limits_{\alpha\to 0+}\frac{(1+\alpha)^2}{4\alpha}\exp{\left[\frac{1-\alpha}{4\alpha}\left[-|u_1|^2 - 
|s_1|^2 +\sqrt{\alpha}\left({u}_1{\bar u}_2-u_2\bar{u}_1+ s_2{\bar s}_1 - {\bar s}_2s_1\right) + \alpha(|{u}_1|^2+|s_2|^2)\right]\right.}
\\
+{\left.\frac{1+\alpha}{4\alpha}\left[u_1{\bar s}_1+{\bar u}_1s_1-\sqrt{\alpha}\left(u_2{\bar s}_1-u_1{\bar s}_2 + {\bar u}_2s_1+{\bar u}_1s_2)\right) +\alpha\left(u_2{\bar s}_2 + {\bar u}s_2\right)\right]\right]}=
\\
=\exp{(-\frac{1}{4}|u_2-s_2|^2)}\lim\limits_{\alpha\to 0+}\frac{(1+\alpha)^2}{4\alpha}\exp{\frac{1}{4\alpha}(|u_1-s_1|^2)}\times
\\
\times\exp{\left[\frac{1}{4\sqrt{\alpha}}\left(-u_1{\bar u}_2+u_2{\bar u}_1-s_2{\bar s}_1-u_2{\bar s}_1+u_1{\bar s}_2+{\bar u}_2s_1+{\bar u}_1s_2\right)\right]}.
\end{multline}
Now we notice that  $\lim\limits_{\alpha\to 0+}{\frac{(1+\alpha)^2}{4\alpha}}\exp{\left[-\frac{1-\alpha}{4\alpha}|u_1-{s}_1|\right.\left.^2\right]}$ {equals}\,\footnote{{\label{miku}\;The equality is meant in the distributional(=week)sense; \cite{mik} might be instrumental in this matter.}} the Dirac measure $\delta(|(u_1-{s}_1)|^2)$ {at $|(u_1-{s}_1)|^2$} which implies that in the argument of the third exponent in \eqref{rep5} we should put $u_1=s_1$.  Further on, the only possibility for which this exponent may be given for $\alpha\to 0$ any meaning  is when the numerator of its argument vanishes identically. This leads to the conditions
\begin{equation}\label{40318_2}
\re{s_1}\re{s_2}-\im{s_1}\im{s_2}=\im{s_2}\re{s_1}-\re{s_2}\im{s_1}=0
\end{equation}
to be satisfied for arbitrary $s_1$. Thus it has to be $\re{s_2}=\im{s_2}=0$, i.e., $s_2=0$. The next observation is that in the limit $\alpha\to 0$ the kernel \eqref{rep5} depends, by definition, only on $u_1=u$ and $s_1=s$ which is enough to set $u_2=0$. All this enable us to conclude\footnote{\; {Cf. footnote \footnoteref{miku}, p. \pageref{miku}.}} that 
\begin{equation}\label{40318_3}
\widetilde{K}^{(0)}(u,s)=\sum\limits_{m,n=0}^{\infty}\widetilde{H}^{(0)}_{m, n}\left({u}, {\bar{u}}\right) \overline{\widetilde{H}^{(0)}_{p, q}}\left(s, \bar{s}\right)=\delta(u-s)\otimes\delta(\overline{u-s})
\end{equation}
which appears in \cite{FanKlauder} under the name of the completeness relation for 2D Hermite polynomials $H_{m,n}(z,\bar z)$; this is extensively analysed in \cite{AWunsche1998,AWunsche2015} as well. 

Let us point out that the spaces in which reproducing kernels act are merely the Hilbert ones \cite{aronszajn,wuj1}; going beyond this framework like it is done above requires to employ a different kind of theory and would make the paper very extensive. This is not our intention here, we would share with the reader very intuitive approach.

\section*{Appendix. Postponed calculations}\label{fhsz:s4.21.03}
{\sc Proof of \eqref{eq4} $\Longleftrightarrow$ \eqref{eq5}.}\label{A1}
Suppose that $m \geq n$ then
\begin{align*}
H_{m, n}(z, \sbar{z}) & = \sum_{k=0}^{n} \binom{m}{k} \binom{n}{k} (-1)^k k! \sum_{r=0}^{m-n} \binom{m-n}{r} x^r (\I\!y)^{m-n-r} \sum_{s=0}^{n-k} \binom{n-k}{s} x^{2s} y^{2(n-k-s)} \nonumber \\
& = \sum_{k=0}^{n} \binom{m}{k} \binom{n}{k} (-1)^k k! (x+\I y)^{m-n} [(x+\I y)(x-\I y)]^{n-k} \nonumber \\
& = \sum_{k=0}^{n} \binom{m}{k} \binom{n}{k} (-1)^k k! (x+\I y)^{m-k} (x-\I y)^{n-k}.
\end{align*}
If $m < n$ then
\begin{align*}
H_{m, n}(z, \sbar{z}) & = \sum_{k=0}^{m} \binom{m}{k} \binom{n}{k} (-1)^k k! \sum_{r=0}^{n-m} \binom{n-m}{r} x^r (-\I\!y)^{n-m-r} \sum_{s=0}^{m-k} \binom{m-k}{s} x^{2s} y^{2(m-k-s)} \nonumber \\
& = \sum_{k=0}^{m} \binom{m}{k} \binom{n}{k} (-1)^k k! (x-\I y)^{n-m} [(x+\I y)(x-\I y)]^{m-k} \nonumber \\
& = \sum_{k=0}^{m} \binom{m}{k} \binom{n}{k} (-1)^k k! (x+\I y)^{m-k} (x-\I y)^{n-k}.\qed
\end{align*}

{\sc Proof of \eqref{eq18}.} \label{A2} For $z=x+\I y=r\E^{\I \theta}$ recalling \eqref{eq20}, write down, 
\begin{multline}\label{eq66}
\int_{\mathbb{R}^{2}} H_{m, n}(x+\I y, x-\I y) \overline{H_{p, q}(x+\I y, x-\I y)} \E^{-x^{2}-y^{2}} \D x \D y \\= \int_{0}^{2\pi}\!\!\int_{0}^{\infty} rH_{m, n}(r\E^{\I\theta}, r\E^{-\I\theta}) \overline{H_{p, q}(r\E^{\I\theta}, r\E^{-\I\theta}))} \E^{-r^{2}} \D \theta\D r.
\end{multline}
Using \eqref{eq15} for LHS of  \eqref{eq66} we get
\begin{align*}
& \text{LHS of \eqref{eq66}} = 2^{-(n+m+q+p)} \sum_{k=0}^{m} \sum_{l=0}^{n} \sum_{i=0}^{p} \sum_{j=0}^{q} \binom{m}{k} \binom{n}{l} \binom{p}{i} \binom{q}{j} \I^{m-k+q-j} (-\I)^{n-l+p-i} \nonumber\\
&\qquad  \times \int_{-\infty}^{\infty} H_{k+l}(x) H_{i+j}(x) \E^{-x^{2}} \D x  \int_{-\infty}^{\infty} H_{m+n-k-l}(y) H_{p + q -i-j}(y) \E^{-y^{2}} \D y, 
\end{align*}
which, due to the normalization for the standard Hermite polynomials in a single variable makes LHS of \eqref{eq66} multiplied by $\pi$. 

Concerning RHS of  \eqref{eq66} we use  \eqref{eq17} combined with \eqref{eq20} to get
\begin{align}
&\text{RHS of \eqref{eq66}} = \sum_{k=0}^{\min\{m, n\}} \sum_{l=0}^{\min\{p,q\}} \binom{m}{k} \binom{n}{k} \binom{p}{l}  \binom{q}{l} (-1)^{k+l} k! l! \int_{0}^{2\pi} \E^{\I\theta(m-p - n + q)} \D\theta \nonumber \\
&\qquad \times \int_{0}^{\infty} r^{1+m+n+p+q-2k-2l} \E^{-r^{2}} \D r \nonumber \\
&\qquad = \pi \sum_{k=0}^{\min\{m,n\}} \binom{m}{k} \binom{n}{k} (-1)^{k} k! \sum_{l=0}^{\min\{m,n\}} \binom{m}{l} \binom{n}{l} (-1)^{l} l! (n+m-k-l)! \delta_{m, p} \delta_{n, q} \nonumber \\
&\qquad = \pi \sum_{k=0}^{\min\{m,n\}} \binom{m}{k} \binom{n}{k} (-1)^{k} k! (m+n-k)!\, {_{2}F_{1}}\left({-m, -n \atop -m-n-k}; 1\right) \delta_{m, p} \delta_{n, q} \label{eq67} \\
& \qquad = \pi \left[(m+n)! {_{2}F_{1}}\left({-m, -n \atop -m-n}; 1\right) \right. \nonumber \\
& \qquad \left.+ \sum_{k=1}^{\min\{m,n\}} \binom{m}{k} \binom{n}{k} (-1)^{k} k! (m+n-k)!\, {_{2}F_{1}}\left({-m, -n \atop -m-n-k}; 1\right)\right] \delta_{m, p} \delta_{n, q} \label{eq68}
\end{align}
In \eqref{eq67} we employ formula (7.3.5.4) of \cite[p. 489]{APPrudnikov-v3-rus} and $(a-k)_{k} = (-1)^{k} (1-a)_{k}$. 

Going on with  \eqref{eq68} we have to use formula (8.334.3) of \cite[p. 896]{Gradshteyn07}  so as to get 
\begin{align}\begin{split}
\eqref{eq68}& =  - m! n! \sum_{k=1}^{\min\{m,n\}} \frac{(-1)^{k}}{k} \frac{\sin[\pi(m-k)] \sin[\pi(n-k)]}{\sin[\pi(m+n-k)]}\, \delta_{m, p} \delta_{n, q} \\&
- m! n! \frac{\cos(\pi m) \cos(\pi n)}{\cos[\pi (m + n)]} \sum_{k=1}^{\min\{m,n\}} \frac{(-1)^{k}\sin(k\pi)}{k} \delta_{m, p} \delta_{n, q}.\end{split}
\label{eq69}
\end{align}

For fixed value of $m$ and $n$  the last sum of \eqref{eq69} becomes $0$, which completes the proof.\qed

{\sc Proof of \eqref{eq22}.}\label{aprilis} Let us begin with 
\begin{lem}\label{l-31/03-4}
For $w=r+\I s$, $r, s\in\rrb$, and $z_{1}, z_{2} \in\ccb$ we have the integral representation
\begin{equation}\label{eq70}
e^{-z_{1} z_{2}} H_{m, n}(z_{1}, z_{2}) = \frac{1}{\pi \I^{m+n}} \int_{\rrb^{2}} w^{m} \sbar{w}^{n} \E^{-w\sbar{w} + \I z_{1} w + \I z_{2} \sbar{w}} \D r \D s.
\end{equation}
\end{lem}
{\em Proof of Lemma.} Substituting $b^{m} = \ulamek{1}{\I^{m}}\ulamek{\D^{m}}{\D a^{m}} \E^{\I a b}$ into RHS of \eqref{eq70} we get
\begin{align*}
\text{RHS of \eqref{eq70}} & = \frac{(-1)^{m + n}}{\pi} \frac{\D^{\; m+n}}{\D z_{1}^{m} \D z_{2}^{n}} \int_{\rrb^{2}} \E^{-w\sbar{w} + \I z_{1} w + \I z_{2} \sbar{w}} \D r\D s \\
& = \frac{(-1)^{m + n}}{\pi} \frac{\D^{\; m+n}}{\D z_{1}^{m} \D z_{2}^{n}} \int_{-\infty}^{\infty}  \E^{-r^{2} + \I(z_{1} + z_{2})r} \D r \int_{-\infty}^{\infty}  \E^{-s^{2} - (z_{1} - z_{2})s} \D s.
\end{align*}
Using \cite[formula (2.3.15.11)]{APPrudnikov-v1-rus} and the Rodrigues formula \eqref{eq9} we can show \eqref{eq70}\,\footnote{\;The proof given in \cite[proof of Theorem 3.1]{ismailzhang} is done for $z$ and $\sbar{z}$ and it does not seem to be applicable here.}. \qed

Now we come back to the proof of \eqref{eq22}. Like \cite[the proof of (2.5)]{ismailzhang} we use Lemma \ref{l-31/03-4} in the proof of \eqref{eq22}. Thus,
\begin{align*}
\pi \E^{-|z_{1}| |z_{2}|} & \Big\vert H_{m, n}(z_{1}, z_{2})\Big\vert = \left\vert \int_{\rrb^{2}} w^{m} \sbar{w}^{n} \E^{-w \sbar{w} + \I z_{1} w + \I z_{2} \sbar{w}} \D r\D s\right\vert
\leq  \int_{\rrb^{2}} |w|^{m+n} \Big\vert \E^{-w \sbar{w} + \I z_{1} w + \I z_{2} \sbar{w}}\Big\vert \D r\D s \\
& = \int_{\rrb^{2}} (r^{2} + s^{2})^{m+n} \E^{-r^{2} - s^{2} - {\rm Im}(z_{1} + z_{2}) r - {\rm Re}(z_{1} - z_{2}) s} \D r\D s \leq \int_{\rrb^{2}} (r^{2} + s^{2})^{m+n} \E^{-r^{2} - s^{2}} \D r\D s.
\end{align*}
From \cite[the proof of (2.5)]{ismailzhang} we have
\begin{equation*}
\pi \E^{-|z_{1}| |z_{2}|} \Big\vert H_{m, n}(z_{1}, z_{2})\Big\vert \leq \pi \sqrt{m! n!}, 
\end{equation*}
which completes the proof of \eqref{eq22}. \qed

{\sc Proof of Remark \ref{t1.12.05}.}\label{rem6}
Because $|{h}^{(\alpha)}_{m, n}(z_{1}, z_{2})|^{2}$ are positive to determine the convergence of  the series \eqref{eq26} we are allowed to rearrange the way of summation. More precisely,  we can split the indices $m$ and $n$ according to $m\geq n$, i.e. $m = n + s$, $s=0, 1, 2, \ldots$, and $m < n$  i.e. $n=m+s$, $s=1, 2, \ldots$. Thus, 
\begin{equation}\label{eq71}
\sum_{m,n=0}^{\infty}|{h}^{(\alpha)}_{m, n}(z_{1}, z_{2})|^{2} = \sum_{s=0}^{\infty}\sum_{n =0}^{\infty} |{h}^{(\alpha)}_{n+s, n}(z_{1}, z_{2})|^{2} + \sum_{s=1}^{\infty}\sum_{m = 0}^{\infty} |{h}^{(\alpha)}_{m, m+s}(z_{1}, z_{2})|^{2}.
\end{equation}

\noindent
With the help of formulae  \eqref{eq20} and $H_{m, n}(z_{1}, z_{2}) = \exp\left(\!-\frac{\partial^{2}}{\partial z_{1} \partial z_{2}}\!\right) z_{1}^{m} z_{2}^{n}$ the first ingredient of the right hand  side of \eqref{eq71}  can be transformed as follows
\begin{align}\label{eq72}
\begin{split}
&\sum_{s=0}^{\infty}\sum_{n=0}^{\infty} |{h}^{(\alpha)}_{n+s, n}(z_{1}, z_{2})|^{2} = \sum_{s=0}^{\infty}\sum_{n=0}^{\infty} {h}^{(\alpha)}_{n+s, n}(z_{1}, z_{2}) \overline{{h}^{(\alpha)}_{n+s, n}(z_{1}, z_{2})} \\
& = \frac{(1-\alpha)^{2}}{\pi^{2} \alpha} \E^{-\ulamek{1}{2}(z_{1}z_{2} + \sbar{z}_{1}\sbar{z}_{2})} \sum_{s=0}^{\infty} \left(\frac{1-\alpha}{1+\alpha}\right)^{s} \sum_{n=0}^{\infty}  \frac{\left(\frac{1-\alpha}{1+\alpha}\right)^{2n}}{n! (n+s)!} H_{n+s, n}(z_{1}, z_{2}) H_{n+s, n}(\sbar{z}_{1}, \sbar{z}_{2})\\
& =  \frac{(1-\alpha)^{2}}{\pi^{2} \alpha} \E^{-\ulamek{1}{2}(z_{1}z_{2} + \sbar{z}_{1}\sbar{z}_{2})}  \E^{-\ulamek{\partial^{2}}{\partial z_{1} \partial z _{2}} -\ulamek{\partial^{2}}{\partial \sbar{z}_{1} \partial \sbar{z} _{2}}}  \sum_{s=0}^{\infty}\left(\frac{1-\alpha}{1+\alpha}\right)^{s} (z_{1}\sbar{z}_{2})^{s} \sum_{n=0}^{\infty} \left(\frac{1-\alpha}{1+\alpha}\right)^{2n}\frac{(z_{1}\sbar{z}_{1} z_{2}\sbar{z}_{2})^{n}}{n! (n+s)!}.
\end{split}
\end{align}

Using \cite[Eq. (8.445) on p. 919]{Gradshteyn07} in the sum over $n$  the formula \eqref{eq72} can be written as
\begin{multline}\label{eq73}
\sum_{s=0}^{\infty}\sum_{n=0}^{\infty} |{h}^{(\alpha)}_{n+s, n}(z_{1}, z_{2})|^{2} = \frac{(1-\alpha)^{2}}{\pi^{2} \alpha} \E^{-\ulamek{1}{2}(z_{1}z_{2} + \sbar{z}_{1}\sbar{z}_{2})}  \E^{-\ulamek{\partial^{2}}{\partial z_{1} \partial z _{2}} -\ulamek{\partial^{2}}{\partial \sbar{z}_{1} \partial \sbar{z} _{2}}} \\ \times \sum_{s=0}^{\infty} \left(\frac{z_{1}\sbar{z}_{1}}{z_{2}\sbar{z}_{2}}\right)^{\frac{s}{2}} I_{s}\left(2\frac{1-\alpha}{1+\alpha} \sqrt{z_{1} \sbar{z}_{1} z_{2} \sbar{z}_{2}}\right),
\end{multline}
where $I_s(x)$ is the modified Bessel function of the first kind. Similarly, the second sum in \eqref{eq71} is 
\begin{align}\label{eq74}
\begin{split}
\sum_{s=1}^{\infty}\sum_{m=0}^{\infty} |{h}^{(\alpha)}_{m, m+s}(z_{1}, z_{2})|^{2} & =  \frac{(1-\alpha)^{2}}{\pi^{2} \alpha} \E^{-\ulamek{1}{2}(z_{1}z_{2} + \sbar{z}_{1}\sbar{z}_{2})}  \E^{-\ulamek{\partial^{2}}{\partial z_{1} \partial z _{2}} -\ulamek{\partial^{2}}{\partial \sbar{z}_{1} \partial \sbar{z} _{2}}} \sum_{s=1}^{\infty} \left(\frac{z_{2}\sbar{z}_{2}}{z_{1}\sbar{z}_{1}}\right)^{\frac{s}{2}} I_{s}\left(2\frac{1-\alpha}{1+\alpha} \sqrt{z_{1} \sbar{z}_{1} z_{2} \sbar{z}_{2}}\right) \\
& = \frac{(1-\alpha)^{2}}{\pi^{2} \alpha} \E^{-\ulamek{1}{2}(z_{1}z_{2} + \sbar{z}_{1}\sbar{z}_{2})}  \E^{-\ulamek{\partial^{2}}{\partial z_{1} \partial z _{2}} -\ulamek{\partial^{2}}{\partial \sbar{z}_{1} \partial \sbar{z} _{2}}} \sum_{s=1}^{\infty} \left(\frac{z_{1}\sbar{z}_{1}}{z_{2}\sbar{z}_{2}}\right)^{-\frac{s}{2}} I_{-s}\left(2\frac{1-\alpha}{1+\alpha} \sqrt{z_{1} \sbar{z}_{1} z_{2} \sbar{z}_{2}}\right).
\end{split}
\end{align}
Combining together \eqref{eq73}, \eqref{eq74}, and \eqref{eq71} we get
\begin{multline*}
\sum_{m=0}^{\infty}\sum_{n=0}^{\infty}|{h}^{(\alpha)}_{m, n}(z_{1}, z_{2})|^{2} = \frac{(1-\alpha)^{2}}{\pi^{2} \alpha} \E^{-\ulamek{1}{2}(z_{1}z_{2} + \sbar{z}_{1}\sbar{z}_{2})}  \E^{-\ulamek{\partial^{2}}{\partial z_{1} \partial z _{2}} -\ulamek{\partial^{2}}{\partial \sbar{z}_{1} \partial \sbar{z} _{2}}} \\ 
\times\sum_{s=-\infty}^{\infty} \left(\frac{z_{1}\sbar{z}_{1}}{z_{2}\sbar{z}_{2}}\right)^{\frac{s}{2}} I_{s}\left(2\frac{1-\alpha}{1+\alpha} \sqrt{z_{1} \sbar{z}_{1} z_{2} \sbar{z}_{2}}\right),
\end{multline*}
from which, after using  \cite[formula (5.8.3.2), p. 694]{APPrudnikov-v2-rus}, we obtain 
\begin{equation}\label{eq75}
\sum_{m=0}^{\infty}\sum_{n=0}^{\infty}|{h}^{(\alpha)}_{m, n}(z_{1}, z_{2})|^{2} = \frac{(1-\alpha)^{2}}{\pi^{2} \alpha} \E^{-\ulamek{1}{2}(z_{1}z_{2} + \sbar{z}_{1}\sbar{z}_{2})}  \E^{-\ulamek{\partial^{2}}{\partial z_{1} \partial z _{2}} -\ulamek{\partial^{2}}{\partial \sbar{z}_{1} \partial \sbar{z} _{2}}} \E^{\ulamek{1-\alpha}{1+\alpha}(z_{1}\sbar{z}_{1} + z_{2}\sbar{z}_{2})}.
\end{equation}
Taking $z_{1} = x_{1} + \I y_{1}$ and $z_{2} = x_{2} + \I y_{2}$ we arrive at
\begin{equation}\label{eq76}
 \E^{-\ulamek{\partial^{2}}{\partial z_{1} \partial z _{2}} -\ulamek{\partial^{2}}{\partial \sbar{z}_{1} \partial \sbar{z} _{2}}} \E^{\ulamek{1-\alpha}{1+\alpha}(z_{1}\sbar{z}_{1} + z_{2}\sbar{z}_{2})} = \E^{-\ulamek{1}{2} \ulamek{\partial^{2}}{\partial x_{1} \partial x_{2}}} \E^{\ulamek{1-\alpha}{1+\alpha}(x_{1}^{2} + x_{2}^{2})}  \E^{\ulamek{1}{2} \ulamek{\partial^{2}}{\partial y_{1} \partial y_{2}}} \E^{\ulamek{1-\alpha}{1+\alpha}(y_{1}^{2} + y_{2}^{2})}.
\end{equation}
\begin{lem}\label{8.09-3}
\begin{equation*}
\left(\frac{\partial^{2}}{\partial x_{1} \partial x_{2}}\right)^{n} \E^{\ulamek{1-\alpha}{1+\alpha}(x_{1}^{2} + x_{2}^{2})} = \left(\frac{1-\alpha}{1+\alpha}\right)^{n} H_{n}\left(\I x_{1}\sqrt{\ulamek{1-\alpha}{1+\alpha}}\right) H_{n}\left(\I x_{2}\sqrt{\ulamek{1-\alpha}{1+\alpha}}\right) \E^{\ulamek{1-\alpha}{1+\alpha}(x_{1}^{2} + x_{2}^{2})}.
\end{equation*}
\end{lem}
\begin{prf}
Let us assume that Lemma \ref{8.09-3} is satisfied for some $n$. Thus,
\begin{multline}
\left(\frac{\partial^{2}}{\partial x_{1} \partial x_{2}}\right)^{n+1} \E^{\ulamek{1-\alpha}{1+\alpha}(x_{1}^{2} + x_{2}^{2})} 
\\
= \left(\frac{1-\alpha}{1+\alpha}\right)^{n} \frac{\partial}{\partial x_{1}} \left[H_{n}\left(\I x_{1}\sqrt{\ulamek{1-\alpha}{1+\alpha}}\right)  \E^{-\ulamek{1-\alpha}{1+\alpha}(\I x_{1})^{2}} \right] 
\frac{\partial}{\partial x_{2}} \left[H_{n}\left(-\I x_{2}\sqrt{\ulamek{1-\alpha}{1+\alpha}}\right)  \E^{-\ulamek{1-\alpha}{1+\alpha}(-\I x_{2})^{2}} \right] 
\end{multline}
which gives RHS. \qed
\end{prf}
\begin{lem}
\begin{align}
\E^{-\ulamek{1}{2} \ulamek{\partial^{2}}{\partial x_{1} \partial x_{2}}} \E^{\ulamek{1-\alpha}{1+\alpha}(x_{1}^{2} + x_{2}^{2})} & = \frac{1+\alpha}{2\sqrt{\alpha}} \E^{\ulamek{1-\alpha^{2}}{4\alpha}(x_{1}^{2} + x_{2}^{2}) - \ulamek{(1-\alpha)^{2}}{2\alpha} x_{1}x_{2}}, \label{eq77} \\
\E^{\ulamek{1}{2} \ulamek{\partial^{2}}{\partial y_{1} \partial y_{2}}} \E^{\ulamek{1-\alpha}{1+\alpha}(y_{1}^{2} + y_{2}^{2})} & = \frac{1+\alpha}{2\sqrt{\alpha}} \E^{\ulamek{1-\alpha^{2}}{4\alpha}(y_{1}^{2} + y_{2}^{2}) + \ulamek{(1-\alpha)^{2}}{2\alpha} y_{1}y_{2}}. \label{eq78}
\end{align}
\end{lem}
\begin{prf}
To prove \eqref{eq78} we use the first formula just below the equation (5.7) in  \cite{Ali2014}, p.  012017-8. \qed
\end{prf}
Thus, \eqref{eq76} have the form
\begin{align*}
\begin{split}
 \E^{-\ulamek{\partial^{2}}{\partial z_{1} \partial z _{2}} -\ulamek{\partial^{2}}{\partial \sbar{z}_{1} \partial \sbar{z} _{2}}} \E^{\ulamek{1-\alpha}{1+\alpha}(z_{1}\sbar{z}_{1} + z_{2}\sbar{z}_{2})} = \frac{(1+\alpha)^{2}}{4\alpha} \E^{\ulamek{1-\alpha^{2}}{4\alpha}(z_{1}\sbar{z}_{1} + z_{2}\sbar{z}_{2}) - \ulamek{(1-\alpha)^{2}}{4\alpha}(z_{1}z_{2} + \sbar{z}_{1} \sbar{z}_{2})}.
\end{split}
\end{align*}
Substituting it into \eqref{eq75} we get \eqref{eq26}.\qed

\end{document}